\documentclass[12pt]{article}

\usepackage{amsmath,enumerate,amsfonts,amssymb,color,graphicx,amsthm,stmaryrd}

\usepackage[normalem]{ulem}

\usepackage{hyperref}

\def\RR{{\mathbb R}}

\def\SSphere{{\mathbb S}}

\def\consta{{c_1}}
\def\constb{{c_2}}

\newcounter{marnote}

\begin{document}
\newtheorem{THM}{Theorem}
\renewcommand*{\theTHM}{\Alph{THM}}
\newtheorem{Def}{Definition}[section]
\newtheorem{thm}[Def]{Theorem}
\newtheorem{lem}[Def]{Lemma}
\newtheorem{question}[Def]{Question}
\newtheorem{prop}[Def]{Proposition}
\newtheorem{cor}[Def]{Corollary}
\newtheorem{clm}[Def]{Claim}
\newtheorem{step}[Def]{Step}
\newtheorem{sbsn}[Def]{Subsection}
\newtheorem{conj}[Def]{Conjecture}
\theoremstyle{remark}
\newtheorem{rem}[Def]{Remark}

\numberwithin{equation}{section}


\title{Solutions to the $\sigma_k$-Loewner-Nirenberg problem on annuli are locally Lipschitz and not differentiable}

\author{YanYan Li \footnote{Department of Mathematics, Rutgers University, Hill Center, Busch Campus, 110 Frelinghuysen Road, Piscataway, NJ 08854, USA. Email: yyli@math.rutgers.edu.}~\footnote{Partially supported by NSF grant DMS-1501004.}~ and Luc Nguyen \footnote{Mathematical Institute and St Edmund Hall, University of Oxford, Andrew Wiles Building, Radcliffe Observatory Quarter, Woodstock Road, Oxford OX2 6GG, UK. Email: luc.nguyen@maths.ox.ac.uk.}}
\date{}
\maketitle

\begin{center}
Dedicated to Alice Chang and Paul Yang on their 70th birthday
\end{center}

\begin{abstract}
We show for $k \geq 2$ that the locally Lipschitz viscosity solution to the $\sigma_k$-Loewner-Nirenberg problem on a given annulus $\{a < |x| < b\}$ is $C^{1,\frac{1}{k}}_{\rm loc}$ in each of $\{a < |x| \leq \sqrt{ab}\}$ and $\{\sqrt{ab} \leq |x| < b\}$ and has a jump in radial derivative across $|x| = \sqrt{ab}$. Furthermore, the solution is not $C^{1,\gamma}_{\rm loc}$ for any $\gamma > \frac{1}{k}$. Optimal regularity for solutions to the $\sigma_k$-Yamabe problem on annuli with finite constant boundary values is also established.
\end{abstract}

\section{Introduction}

Let $\Omega$ be a smooth bounded domain in $\RR^n$, $n \geq 3$. For a positive $C^2$ function $u$ defined on an open subset of $\RR^n$, let $A^u$ denote its conformal Hessian, namely
\begin{equation}\label{conformal-Hessian}
A^u = -\frac{2}{n-2} u^{-\frac{n+2}{n-2}}\,\nabla^2 u + \frac{2n}{(n-2)^2} u^{-\frac{2n}{n-2}} \nabla u \otimes \nabla u - \frac{2}{(n-2)^2}\,u^{-\frac{2n}{n-2}}\,|\nabla u|^2\,I,
\end{equation}
and let $\lambda(-A^u)$ denote the eigenvalues of $-A^u$. Note that $A^u$, considered as a $(0,2)$ tensor, is  the Schouten curvature tensor of the metric  $u^{\frac{4}{n-2}}\mathring{g}$ where $\mathring{g}$ is the Euclidean metric.

For $1 \leq k \leq n$, let $\sigma_k: \RR^n \rightarrow \RR$ denote $k$-th elementary symmetric function
\[
\sigma_k(\lambda)=\sum_{i_1<\ldots<i_k} \lambda_{i_1}\ldots\lambda_{i_k}, 
\]
and let $\Gamma_k$ denote the cone $\Gamma_k=\{\lambda=(\lambda_1,\ldots,\lambda_n):\sigma_1(\lambda)>0,\ldots,\sigma_k(\lambda)>0\}$. 

In \cite[Theorem 1.1]{GonLiNg}, it was shown that the $\sigma_k$-Loewner-Nirenberg problem
\begin{align}
\sigma_k(\lambda(-A^u)) 
	&= 2^{-k} \Big(\begin{array}{c}n\\k\end{array}\Big), \qquad \lambda(-A^u) \in \Gamma_k, \qquad u  > 0 \qquad  \text{ in } \Omega
	\label{Eq:X1},\\
u(x) &\rightarrow \infty  \text{ as } \textrm{dist}(x,\partial\Omega) \rightarrow 0.
	\label{Eq:X1BC}
\end{align}
has a unique continuous viscosity solution $u$ and such $u$ belongs to $C^{0,1}_{\rm loc}(\Omega)$. Furthermore, $u$ satisfies
\begin{equation}
\lim_{d(x,\partial\Omega) \rightarrow 0} u(x) d(x,\partial\Omega)^{-\frac{n-2}{2}}= C(n,k) \in (0,\infty).
	\label{Eq:08IV20-E1}
\end{equation}

Equation \eqref{Eq:X1} is a fully nonlinear elliptic equation of the kind considered by Caffarelli, Nirenberg and Spruck \cite{C-N-S-Acta}. We recall the following definition of viscosity solutions which follows Li \cite[Definitions 1.1 and 1.1']{Li09-CPAM} (see also \cite{Li06-Arxivv2}) where viscosity solutions were first considered in the study of nonlinear Yamabe problems.

Let
\begin{align}
\overline{S}_k 
	&:= \Big\{\lambda \in\Gamma_k \big|  \sigma_k(\lambda) \geq 2^{-k} \Big(\begin{array}{c}n\\k\end{array}\Big)\Big\},
	\label{Eq:oSkDef}\\
\underline{S}_k
	&:= \RR^n \setminus \Big\{\lambda \in \Gamma_k\Big| \sigma_k(\lambda) > 2^{-k} \Big(\begin{array}{c}n\\k\end{array}\Big)\Big\}.
	\label{Eq:uSkDef}
\end{align}

\begin{Def}\label{Def:ViscositySolution}
Let $\Omega\subset\mathbb{R}^{n}$ be an open set and $1 \leq k \leq n$. We say that an upper semi-continuous (a lower semi-continuous) function $u: \Omega \rightarrow (0,\infty)$ is a sub-solution (super-solution) to \eqref{Eq:X1} in the viscosity sense, if for any $x_{0}\in\Omega$, $\varphi\in C^{2}(\Omega)$ satisfying $(u-\varphi)(x_{0})=0$ and $u-\varphi\leq0$ $(u-\varphi\geq0)$ near $x_{0}$, there holds
\[
\lambda\big(-A^\varphi(x_{0})\big)\in \overline{S}_k \qquad
 \left(\lambda\big(-A^\varphi(x_{0})\big) \in \underline{S}_k, \text{respectively}\right).
\]

We say that a positive function $u \in C^0(\Omega)$ satisfies \eqref{Eq:X1} in the viscosity sense if it is both a sub- and a super-solution to \eqref{Eq:X1} in the viscosity sense.
\end{Def}

Equation \eqref{Eq:X1} satisfies the following comparison principle, which is a consequence of the principle of propagation of touching points \cite[Theorem 3.2]{LiNgWang}: If $v$ and $w$ are viscosity sub-solution and super-solution of \eqref{Eq:X1} and if $v \leq w$ near $\partial\Omega$, then $v \leq w$ in $\Omega$; see \cite[Proposition 2.2]{GonLiNg}. The above mentioned uniqueness result for \eqref{Eq:X1}-\eqref{Eq:X1BC} is a consequence of this comparison principle and the boundary estimate \eqref{Eq:08IV20-E1}.

In the rest of this introduction, we assume that  $\Omega$ is an annulus $\{a < |x|< b\} \subset \RR^n$ with $0 < a < b < \infty$, unless otherwise stated. $C^2$ radially symmetric solutions to \eqref{Eq:X1} were classified by Chang, Han and Yang \cite[Theorems 1 and 2]{C-H-Y}. As a consequence, when $2 \leq k \leq n$, there is no $C^2$ radially symmetric function satisfying \eqref{Eq:X1}-\eqref{Eq:X1BC}. On the other hand, the aforementioned uniqueness result from \cite{GonLiNg, LiNgWang} implies that the solution $u$ to \eqref{Eq:X1}-\eqref{Eq:X1BC} is radially symmetric (since $u(R \cdot)$ is also a solution for any orthogonal matrix $R$). Therefore, \eqref{Eq:X1}-\eqref{Eq:X1BC} has no $C^2$ solutions. 

Our first result improves on the above non-existence of $C^2$ solutions to \eqref{Eq:X1}-\eqref{Eq:X1BC}.

\begin{thm}\label{Thm:AnnularNoSub}
Suppose that $n \geq 3$ and $\Omega$ is a non-empty open subset of $\RR^n$. Then there exists no positive function $\underline{u} \in C^2(\Omega)$ such that $\lambda(-A^{\underline{u}}) \in \bar\Gamma_2$ in $\Omega$ and that $(\Omega, u^{\frac{4}{n-2}}\mathring{g})$ admits a smooth minimal immersion $f: \Sigma^{n-1} \rightarrow \Omega$ for some smooth compact manifold $\Sigma^{n-1}$.
\end{thm}

Theorem \ref{Thm:AnnularNoSub} bears some resemblance to a result of Schoen and Yau \cite{SchoenYau-AnnM79} on a relation between positive scalar curvature and stable minimal surfaces.

Noting that when $u$ is radially symmetric, $\partial B_{r_0}$ is minimal with respect to $u^{\frac{4}{n-2}}\mathring{g}$ if and only if $\frac{d}{dr}\big|_{r = r_0} (r^{\frac{n-2}{2}} u(r)) = 0$, we obtain the following corollary with $r_0$ being a minimum point of $r^{\frac{n-2}{2}} u(r)$.

\begin{cor}
Suppose that $n \geq 3$. Let $\Omega = \{a < |x| < b\} \subset \RR^n$ with $0 < a < b < \infty$ be an annulus. Then there exists no radially symmetric positive function $\underline{u} \in C^2(\Omega)$ such that $\lambda(-A^{\underline{u}}) \in \bar\Gamma_2$ in $\Omega$ and $u(x) \rightarrow \infty$ as $x \rightarrow \partial \Omega$.
\end{cor}

Our next result shows that the locally Lipschitz solution $u$ is not $C^1$.

\begin{thm}\label{thm:MainLipNeg}
Suppose that $n \geq 3$ and $2 \leq k \leq n$.
Let $\Omega = \{a < |x| < b\} \subset \RR^n$ with $0 < a < b < \infty$ be an annulus and $u$ be the unique locally Lipschitz viscosity solution to \eqref{Eq:X1}-\eqref{Eq:X1BC}. Then $u$ is radially symmetric, i.e. $u(x) = u(|x|)$, 
\begin{enumerate}[(i)]
\item $u$ is smooth in each of $\{a < |x| < \sqrt{ab}\}$ and $\{\sqrt{ab} < |x| < b\}$,
\item $u$ is $C^{1,\frac{1}{k}}$ but not $C^{1,\gamma}$ with $\gamma > \frac{1}{k}$ in each of $\{a < |x| \le \sqrt{ab}\}$ and $\{\sqrt{ab} \leq |x| < b\}$,
\item and the first radial derivative $\partial_r u$ jumps across $\{|x| = \sqrt{ab}\}$:
\[
\partial_r \ln u\big|_{r = \sqrt{ab}^-} = -\frac{n-2}{\sqrt{ab}} \text{ and } \partial_r \ln u\big|_{r = \sqrt{ab}^+} = 0.
\]
\end{enumerate}
\end{thm}

A related problem in manifold settings is to solve on a given closed Riemannian manifold $(M,g)$ the equation
\begin{equation}
\sigma_k\Big(\lambda\Big(-A_{u^{\frac{4}{n-2}}g}\Big)\Big) 
	= 2^{-k} \Big(\begin{array}{c}n\\k\end{array}\Big), \quad \lambda\Big(-A_{u^{\frac{4}{n-2}}g}\Big) \in \Gamma_k, \quad u  > 0 \quad  \text{ in } M,
		\label{Eq:X1Mnfd}
\end{equation}
where $A_{u^{\frac{4}{n-2}}g}$ is the so-called Schouten tensor of the metric $u^{\frac{4}{n-2}}g$,
\[
A_{u^{\frac{4}{n-2}}g} = -\frac{2}{n-2} u^{-1}\,\nabla_g^2 u + \frac{2n}{(n-2)^2} u^{-2} d u \otimes d u - \frac{2}{(n-2)^2}\,u^{-\frac{2n}{n-2}}\,|\nabla_g u|_g^2\,g + A_g,
\]
and where $\lambda\big(-A_{u^{\frac{4}{n-2}}g}\big)$ is the eigenvalue of $-A_{u^{\frac{4}{n-2}}g}$ with respect to the metric $u^{\frac{4}{n-2}}g$. Equations \eqref{Eq:X1} and \eqref{Eq:X1Mnfd} are fully non-linear and non-uniformly elliptic equations of Hessian type, usually referred to as the $\sigma_k$-Yamabe equation in the `negative case', which is a generalization of the Loewner-Nirenberg problem \cite{LoewnerNirenberg}. This equation and its variants have been studied in Chang, Han and Yang \cite{C-H-Y}, Gonzalez, Li and Nguyen \cite{GonLiNg}, Gurksy and Viaclovsky \cite{Gursky-Viaclovsky:negative-curvature}, Li and Sheng \cite{Li-Sheng:flow}, Guan \cite{Guan:negative-Ricci}, Gursky, Streets and Warren \cite{Gursky-Streets-Warren}, and Sui \cite{Sui}. For further studies on the counterpart of \eqref{Eq:X1} in the positive case, see \cite{CGY02-AnnM, GeWang06, GW03-IMRN, GV07, LiLi03,LiLi05, Li09-CPAM, LiNgPoorMan, STW07, TW09, Viac00-Duke} and the references therein.

We observe the following result, which is essentially due to Gursky and Viaclovsky \cite{Gursky-Viaclovsky:negative-curvature}. We provide in the appendix the detail for the piece which is not directly available from \cite{Gursky-Viaclovsky:negative-curvature}.

\begin{thm}
\label{Thm:GV}
Suppose that $n \geq 3$, $2 \leq k \leq n$, and $(M^n,g)$ is a compact Riemannian manifold. If $\lambda(-A_g) \in \Gamma_k$ on $M$, then \eqref{Eq:X1Mnfd} has a Lipschitz viscosity solution.
\end{thm}

Here viscosity solution is defined analogously as in Definition \ref{Def:ViscositySolution}.

\begin{Def}\label{Def:ViscositySolutionMnfd}
Let $(M^n,g)$ be a Riemannian manifold, $1 \leq k \leq n$, and $\overline{S}_k$ and $\underline{S}_k$ be given by \eqref{Eq:oSkDef} and \eqref{Eq:uSkDef}. We say that an upper semi-continuous (a lower semi-continuous) function $u: M \rightarrow (0,\infty)$ is a sub-solution (super-solution) to \eqref{Eq:X1Mnfd} in the viscosity sense, if for any $x_{0}\in M$, $\varphi\in C^{2}(M)$ satisfying $(u-\varphi)(x_{0})=0$ and $u-\varphi\leq0$ $(u-\varphi\geq0)$ near $x_{0}$,
there holds
\[
\lambda\Big(-A_{\varphi^{\frac{4}{n-2}}g}(x_{0})\Big)\in \overline{S}_k
\qquad
 \left(\lambda\Big(-A_{\varphi^{\frac{4}{n-2}}g}(x_{0})\Big) \in \underline{S}_k, \text{respectively}\right).
\]

We say that a positive function $u \in C^0(M)$ satisfies \eqref{Eq:X1Mnfd} in the viscosity sense if it is both a sub- and a super-solution to \eqref{Eq:X1Mnfd} in the viscosity sense.
\end{Def}

In both contexts, it is an interesting open problem to understand relevant conditions on $\Omega$, or on $(M,g)$, which would ensure that \eqref{Eq:X1}-\eqref{Eq:X1BC}, or \eqref{Eq:X1Mnfd} respectively, admits a smooth solution. We make the following conjecture.

\begin{conj}\label{C:11IV20-C1}
Suppose that $n \geq 3$, $2 \leq k \leq n$, and $\Omega \subset \RR^n$ is a bounded smooth domain. Then the locally Lipschitz viscosity solution to \eqref{Eq:X1}-\eqref{Eq:X1BC} is smooth near $\partial\Omega$. 
\end{conj}

Some further questions are in order.

\begin{question}\label{Q:7IV20-QA}
Suppose that $n \geq 3$, $2 \leq k \leq n$, and $\Omega \subset \RR^n$ is a bounded smooth domain. If \eqref{Eq:X1}-\eqref{Eq:X1BC} has a smooth sub-solution, must \eqref{Eq:X1}-\eqref{Eq:X1BC} have a smooth solution? 
\end{question}

\begin{question}
Suppose that $n \geq 3$, $2 \leq k \leq n$, and $\Omega \subset \RR^n$ is a smooth strictly convex (non-empty) domain. Is the locally Lipschitz viscosity solution to \eqref{Eq:X1}-\eqref{Eq:X1BC} smooth?
\end{question}

If $\Omega$ is a ball, then the solution to \eqref{Eq:X1}-\eqref{Eq:X1BC} is smooth and corresponds to the Poincar\'e metric.

\begin{question}\label{Q:09IVQ1.9}
Suppose that $n \geq 3$, $2 \leq k \leq n$, and $\Omega = \Omega_2 \setminus \bar \Omega_1 \neq \emptyset$ where $\Omega_1  \Subset \Omega_2 \subset \RR^n$ are smooth bounded strictly convex domains. Is the locally Lipschitz viscosity solution to \eqref{Eq:X1}-\eqref{Eq:X1BC} $C^2$?
\end{question}

In the case $\Omega_1$ and $\Omega_2$ are balls, $\Omega = \Omega_2 \setminus \Omega_1$ is conformally equivalent to an annulus, and so, by Theorem \ref{thm:MainLipNeg}, the solution to \eqref{Eq:X1}-\eqref{Eq:X1BC} is not $C^2$. We believe that the answer to the above question is negative. We indicate here how such statement may be proved. In view of Theorem \ref{Thm:AnnularNoSub}, it suffices to show that if $u$ is a $C^2$ solution to \eqref{Eq:X1}-\eqref{Eq:X1BC}, then $(\Omega, u^{\frac{4}{n-2}}\mathring{g})$ admits a smooth (immersed) minimal hypersurface. It is reasonable to expect, in view of known results in the case $k = 1$ (cf. \cite{Andersson-Chrusciel-Friedrich, Mazzeo:singular-Yamabe}) and estimate \eqref{Eq:08IV20-E1}, that
\[
d(x,\partial\Omega)\Big|\nabla \big(u(x) d(x,\partial\Omega)^{\frac{n-2}{2}}\big)\Big| \rightarrow 0 \text{ as } d(x,\partial\Omega) \rightarrow 0.
\]
If the above estimate holds for $k \geq 2$, one has that, for small $\delta > 0$, the hypersurfaces $X_\delta = \{x \in \Omega: d(x,\partial\Omega) = \delta\}$ are strictly mean-convex with respect to $u^{\frac{4}{n-2}}\mathring{g}$ and the normal pointing toward the region enclosed between these two hypersurfaces. These hypersurfaces can be used as barriers to construct a desired minimal hypersurface, at least for $n \leq 7$. For example, in dimension $n = 3$, a result of Meeks and Yau \cite[Theorem 7]{Meeks-Yau-AnnM80} (see also \cite[Theorem 4.2]{Hass-Scott-TrAMS88}) implies that there exists a conformal map $f: \SSphere^2 \rightarrow \Omega_\delta = \{x \in \Omega: d(x,\partial\Omega) > \delta\}$ which minimizes area among all homotopically nontrivial maps from $\SSphere^2$ into $\Omega_\delta$ and either $f$ is a conformal embedding or a double covering map whose image is an embedded projective plane. Since all compact surfaces in $\RR^3$ are orientable (see e.g. \cite{Samelson-ProcAMS69} or \cite[Corollary 3.46]{Hatcher}), $f$ is a conformal embedding and so $f(\SSphere^2)$ is an embedded minimal sphere in $(\Omega_\delta, u^4 \mathring{g})$. This will be followed up in a subsequent joint work with Jingang Xiong.

\begin{question}
Suppose that $n \geq 3$, $2 \leq k \leq n$, and $(M^n,g)$ is a Riemannian manifold such that $\lambda(-A_g) \in \Gamma_k$ on $M$. Does \eqref{Eq:X1Mnfd} have a unique Lipschitz viscosity solution?
\end{question}

It is clear that \eqref{Eq:X1Mnfd} has at most one $C^2$ solution by the maximum principle. In fact, if \eqref{Eq:X1Mnfd} has a $C^2$ solution, then that solution is also the unique continuous viscosity solution in view of the strong maximum principle \cite[Theorem 3.1]{CafLiNir11}. Equivalently, if \eqref{Eq:X1Mnfd} has two viscosity solutions, then it has no $C^2$ solution.

\begin{question}\label{Q:1IV20-Q1}
Suppose that $n \geq 3$ and $2 \leq k \leq n$. Does there exist a Riemannian manifold $(M^n,g)$ such that $\lambda(-A_g) \in \Gamma_k$ on $M$ and \eqref{Eq:X1Mnfd} has a Lipschitz viscosity solution which is not $C^2$?
\end{question}

Finally, we discuss the case where \eqref{Eq:X1BC} is replaced by finite constant boundary conditions 
\begin{equation}
u|_{\{|x| = a\}} = \consta \text{ and }u|_{\{|x| = b\}} = \constb.
	\label{Eq:X1FiniteBC}
\end{equation}
We completely determine in the following theorem the regularity of the solution to \eqref{Eq:X1} and \eqref{Eq:X1FiniteBC} depending on whether $\ln \frac{b}{a}$ is larger, equal to, or smaller than $2T(a,b,\consta,\constb)$ where
\begin{equation}
T(a,b,\consta,\constb) := \frac{1}{2}\int_{-|p_b - p_a|}^{0}
	\Big\{1 + e^{-2\eta - 2\max(p_a,p_b)}\big[1 - e^{n\eta}\big]^{1/k}\Big\}^{-1/2}\,d\eta,
	\label{Eq:Tabcab}
\end{equation}
$p_a =- \frac{2}{n-2}\ln \consta - \ln a$ and $p_b =- \frac{2}{n-2}\ln \constb - \ln b$.

\begin{thm}\label{thm:MainLipNegCab}
Suppose that $n \geq 3$ and $2 \leq k \leq n$. Let $\Omega = \{a < |x| < b\} \subset \RR^n$ with $0 < a < b < \infty$ be an annulus, and $\consta, \constb$ be two positive constants and let $T(a,b,\consta,\constb)$ be given by \eqref{Eq:Tabcab}. Then there exists a unique continuous viscosity solution to \eqref{Eq:X1} and \eqref{Eq:X1FiniteBC}. Furthermore, $u$ is radially symmetric, i.e. $u(x) = u(|x|)$, and exactly one of the following four alternatives holds.
\begin{enumerate}[{Case} 1:]
\item $\ln \frac{b}{a} < 2T(a,b,\consta,\constb)$, and $u$ is smooth in $\{a \leq |x| \leq b\}$,
\item $\ln \frac{b}{a} = 2T(a,b,\consta,\constb)$, $b^{\frac{n-2}{2}}\constb < a^{\frac{n-2}{2}}\consta$, and $u$ is smooth in $\{a \leq |x| < b\}$, is $C^{1,\frac{1}{k}}$ but not $C^{1,\gamma}$ with $\gamma > \frac{1}{k}$ in $\{a \leq |x| \le b\}$,
\item $\ln \frac{b}{a} = 2T(a,b,\consta,\constb)$, $b^{\frac{n-2}{2}}\constb > a^{\frac{n-2}{2}}\consta$, and $u$ is smooth in $\{a < |x| \leq b\}$, is $C^{1,\frac{1}{k}}$ but not $C^{1,\gamma}$ with $\gamma > \frac{1}{k}$ in $\{a \leq |x| \le b\}$,

\item $\ln \frac{b}{a} > 2T(a,b,\consta,\constb)$, and there is some $m \in (a,b)$ such that
\begin{enumerate}[(i)]
\item $u$ is smooth in each of $\{a \leq |x| < m\}$ and $\{m < |x| \leq b\}$, 
\item $u$ is $C^{1,\frac{1}{k}}$ but not $C^{1,\gamma}$ with $\gamma > \frac{1}{k}$ in each of $\{a \leq |x| \le m\}$ and $\{m \leq |x| \leq b\}$, 
\item and the first radial derivative $\partial_r u$ jumps across $\{|x| = m\}$:
\[
\partial_r \ln u\big|_{r = m^-} = -\frac{n-2}{m} \text{ and } \partial_r \ln u\big|_{r = m^+} = 0.
\]
\end{enumerate}
\end{enumerate}
\end{thm}

Note that when $\ln \frac{b}{a} = 2T(a,b,\consta,\constb)$, we have in view of the definition of $T(a,b,\consta,\constb)$, $p_a$ and $p_b$ that $b^{\frac{n-2}{2}}\constb \neq a^{\frac{n-2}{2}}\consta$.

\begin{rem}\label{Rmk:C1=>Smooth}
It is clear from Theorem \ref{thm:MainLipNegCab} (in Cases 1--3) that if $u$ is a $C^1$ and radially symmetric solution to \eqref{Eq:X1} in the viscosity sense in some open annulus $\Omega$ then $u \in C^\infty(\Omega)$.
\end{rem}

\begin{rem}
In Case 4, the exact value of $m$ is
\[
m = \sqrt{ab} \exp\Big(\frac{1}{2}\int_{p_b- p}^{p_a - p} \Big\{1 + e^{-2\eta-2p}\big[1 - e^{n\eta}\big]^{1/k}\Big\}^{-1/2}\,d\eta\Big)
\]
where $p$ is the solution to
\begin{align*}
\ln\frac{b}{a} 
	&= \int_{p_b- p}^{0} \Big\{1 + e^{-2\eta-2p}\big[1 - e^{n\eta}\big]^{1/k}\Big\}^{-1/2}\,d\eta\\
		&\quad + \int_{p_a- p}^{0} \Big\{1 + e^{-2\eta-2p}\big[1 - e^{n\eta}\big]^{1/k}\Big\}^{-1/2}\,d\eta.
\end{align*}
\end{rem}

The following question is related to Question \ref{Q:7IV20-QA}.

\begin{question}
Suppose that $n \geq 3$, $2 \leq k \leq n$ and $\Omega = \{a < |x| < b\} \subset \RR^n$ with $0 < a < b < \infty$. Does there exist constants $\consta, \constb$ with $\ln \frac{b}{a} > 2T(a,b,\consta,\constb)$ such that the problem \eqref{Eq:X1} and \eqref{Eq:X1FiniteBC} has a smooth sub-solution?
\end{question}

Recall that by Theorem \ref{thm:MainLipNegCab}, when $\ln \frac{b}{a} > 2T(a,b,\consta,\constb)$, the problem \eqref{Eq:X1} and \eqref{Eq:X1FiniteBC} has no smooth solution.

For comparison, we recall here a result of Bo Guan \cite{Guan07-AJM} on the Dirichlet $\sigma_k$-Yamabe problem in the so-called positive case which states that the existence of a smooth sub-solution implies the existence of a smooth solution.

We conclude the introduction with one more question.

\begin{question}
Let $n \geq 3$, $2 \leq k \leq n$ and $m \neq n-1$. Does there exist a smooth domain $\Omega \subset \RR^n$ such that the locally Lipschitz solution to \eqref{Eq:X1}-\eqref{Eq:X1BC} is $C^2$ away from a set $\Sigma$ which has Hausdorff dimension $m$?  
\end{question}

In Section \ref{Sec:Proofs}, we prove all the results above except Theorem \ref{Thm:GV}, whose proof is done in the appendix. Theorem \ref{Thm:AnnularNoSub} is proved first in Subsection \ref{ssec:ANS}. We then prove a lemma on the existence and uniqueness a non-standard boundary value problem for the ODE related to \eqref{Eq:X1} in Subsection \ref{ssec:lemma} and use it to prove Theorem \ref{thm:MainLipNeg} in Subsection \ref{ssec:Main} and Theorem \ref{thm:MainLipNegCab} in Subsection \ref{ssec:Cab}.

\subsection*{Acknowledgment} The authors would like to thank Matt Gursky and Zheng-Chao Han for stimulating discussions. The authors are grateful to the referees for their very careful reading and useful comments.

\section{Proofs}\label{Sec:Proofs}

\subsection{Proof of Theorem \ref{Thm:AnnularNoSub}}
\label{ssec:ANS}

We will use the following lemma.

\begin{lem}\label{Lem:08IV20-L1}
For every symmetric $n \times n$ matrix $M$ with $\lambda(M) \in \bar\Gamma_2$ and every unit vector $m \in \RR^n$, it holds
\[
M_{ij}(\delta_{ij} - m_i m_j) \geq 0.
\]
\end{lem}

\begin{proof}
Using an orthogonal transformation, we may assume without loss of generality that $M$ is diagonal with diagonal entries $\lambda_1 \leq \lambda_2 \leq \ldots \leq \lambda_n$. Then $(\lambda_1, \ldots, \lambda_n) \in \bar\Gamma_2$. It is well known that this implies $\lambda_1 + \ldots + \lambda_{n-1} \geq 0$. Now as 
\[
M_{ij}(\delta_{ij} - m_i m_j)
	= \sum_{\ell = 1}^n \lambda_\ell - \sum_{\ell = 1}^n \lambda_\ell m_\ell^2
	\geq \sum_{\ell = 1}^n \lambda_\ell - \lambda_n \sum_{\ell = 1}^n m_\ell^2 
	= \sum_{\ell = 1}^{n-1} \lambda_\ell,
\]
the conclusion follows.
\end{proof}

We will use the following result on the mean curvatures of an immersed hypersurface with respect to two conformal metrics. Let $\Omega \subset \RR^n$ be an open set. Equip $\Omega$ with the Euclidean metric $\mathring{g}$ and a conformal metric $\mathring{g}_u := u^{\frac{4}{n-2}}\mathring{g}$ where $u$ is $C^2$. Let $f: \Sigma^{n-1} \rightarrow \Omega$ be a smooth immersion of a compact manifold $\Sigma^{n-1}$ into $\Omega$. Let $\tilde u = u \circ f$, $\tilde g = f^* \mathring{g}$. For every point $p \in \Sigma$, let $H_{\Sigma}(p)$ and $H_{\Sigma,u}(p)$ denote the mean curvature vectors associated to $f$ at $f(p)$ and with respect to $\mathring{g}$ and $\mathring{g}_u$, respectively. To dispel confusion, we note that, in our notation, the mean curvature is the trace of the second fundamental form. Note that if $\nu$ is a unit vector at $f(p)$ normal to the image of a small neighborhood of $p$, then
\begin{equation}
\partial_\nu u(f(p))  + \frac{n-2}{2(n-1)} \mathring{g}(H_\Sigma(p),\nu) u(f(p)) =  \frac{n-2}{2(n-1)} \mathring{g}_u(H_{\Sigma,u}(p), u^{-\frac{2}{n-2}} \nu) u^{\frac{n}{n-2}}.
	\label{Eq:06IV20-X1}
\end{equation}

\begin{lem}\label{Lem:08IV20-L2}
Let $\Omega$ be an open subset of $\RR^n$, $n \geq 3$, and $f:\Sigma^{n-1} \rightarrow \Omega$ be a smooth immersion. If $u \in C^2(\Omega)$ satisfies $\lambda(-A^u) \in \bar\Gamma_2$ in $\Omega$, then
\[
 \Delta_{\tilde g} \tilde u + \frac{n-2}{4(n-1)} |H_{\Sigma,u}|_{\mathring{g}_u}^2 \tilde u^{\frac{n+2}{n-2}} - \frac{n-2}{4(n-1)} |H_\Sigma|_{\mathring{g}}^2 \tilde u   - \frac{1}{(n-2) \tilde u}\,|\nabla_{\tilde g} \tilde u|^2
 	\geq 0 \text{ on } \Sigma.
\]
\end{lem}

\begin{proof} Fix some $p \in \Sigma^{n-1}$ and let $\nu$ be a unit vector at $f(p)$ normal to the image of a small neighborhood of $p$,. Recall that
\[
A^u = -\frac{2}{n-2} u^{-\frac{n+2}{n-2}}\Big[\nabla^2 u - \frac{n}{(n-2)u} \nabla u \otimes \nabla u 
	+  \frac{1}{n-2} |\nabla u|^2\,I\Big].
\]
Applying Lemma \ref{Lem:08IV20-L1} with $M = - \frac{n-2}{2} u^{\frac{n+2}{n-2}}A^u(f(p))$ and $m = \nu$ yields
\[
0 
	\leq \nabla_i\,\nabla_j u\,\big(\delta_{ji} - \nu_i \nu_j\big) - \frac{1}{(n-2)u} \,|\nabla u|^2 + \frac{n}{(n-2)u} \,|\partial_\nu u|^2.
\]
This means
\[
0
	\leq \Delta_{\tilde g} \tilde u + \mathring{g}(H_\Sigma, \nu) \partial_\nu u \circ f  + \frac{n-1}{(n-2)\tilde u} \,|\partial_\nu u \circ f|^2 - \frac{1}{(n-2)\tilde u}\,|\nabla_{\tilde g} \tilde u|^2 \text{ on } \Sigma.
\]
Using \eqref{Eq:06IV20-X1} yields the conclusion.
\end{proof}

\begin{proof}[Proof of Theorem \ref{Thm:AnnularNoSub}]
Suppose by contradiction that $u \in C^2(\Omega)$ is such that $\lambda(-A^u) \in \bar\Gamma_2$ in $\Omega$ and $(\Omega, u^{\frac{4}{n-2}} \mathring{g})$ admits a smooth minimal immersion $f: \Sigma^{n-1} \rightarrow \Omega$ for some smooth compact manifold $\Sigma^{n-1}$. Here we have renamed $\underline{u}$ in the statement of the theorem as $u$ for notational convenience. Let $\nu$ denote a continuous unit normal along $\Sigma$. By Lemma \ref{Lem:08IV20-L2}, we have
\begin{align*}
 \Delta_{\tilde g} \tilde u - \frac{n-2}{4(n-1)} |H_\Sigma|_{\mathring{g}}^2 \tilde u   - \frac{1}{(n-2) \tilde u}\,|\nabla_{\tilde g} \tilde u|^2 \geq 0 \text{ on } \Sigma.
\end{align*}
Integrating over $\Sigma$, we thus have that $H_\Sigma \equiv 0$ and $\tilde u \equiv \textrm{const}$ on $\Sigma$. In particular, $f: \Sigma^{n-1} \rightarrow \Omega$ is a minimal immersion with respect to $\mathring{g}$. This is impossible as there is no smooth minimal immersion in $\RR^n$ with codimension one.
\end{proof}

\subsection{Preliminary ODE analysis}

By the uniqueness result in \cite{GonLiNg, LiNgWang}, the solutions $u$ in Theorems \ref{thm:MainLipNeg} and \ref{thm:MainLipNegCab} are radially symmetric, $u(x) = u(r)$ where $r = |x|$.

As in \cite{C-H-Y, Viac00-Duke}, we work on a round cylinder instead of $\RR^n$. Namely, let
\[
t
	= \ln r - \frac{1}{2} \ln(ab), \qquad 
\xi(t)	= -\frac{2}{n-2}\ln u(r) - \ln r
\]
so that $u^{\frac{4}{n-2}}\mathring{g} = e^{-2\xi} (dt^2 + g_{\mathbb{S}^{n-1}})$. A direct computation gives that, at points where $u$ is twice differentiable,
\begin{equation}
\sigma_k(\lambda(-A^u))
	= \frac{(-1)^k}{2^{k-1}}\,\Big(\begin{array}{c} n-1\\ k-1\end{array}\Big)\,e^{2k\xi}(1 - |\xi'|^2)^{k-1}[\xi'' + \frac{n-2k}{2k}(1-|\xi'|^2)],
\label{Eq:sigmakExpr}
\end{equation}
where here and below $'$ denotes differentiation with respect to $t$. 

Note that, for $k \geq 2$, at points where $u$ is twice differentiable, $\lambda(-A^u) \in \Gamma_k$ if and only if $\sigma_k(\lambda(-A^u)) > 0$ and $|\xi'| > 1$. Indeed, if $\sigma_k(\lambda(-A^u)) > 0$ and $|\xi'| > 1$, then \eqref{Eq:sigmakExpr} implies $\sigma_i(\lambda(-A^u)) > 0$ for $1 \leq i \leq k$ and so $\lambda(-A^u) \in \Gamma_k$. Conversely, if $\lambda(-A^u) \in \Gamma_k$ for some $k \geq 2$, then $\sigma_1(\lambda(-A^u)) > 0$, $\sigma_2(\lambda(-A^u)) > 0$ and $\sigma_k(\lambda(-A^u)) > 0$. Using \eqref{Eq:sigmakExpr}, we see that the first two inequalities imply $|\xi'| > 1$. 

By the same reasoning, we have, at points where $u$ is twice differentiable, if $\lambda(-A^u) \in \bar\Gamma_2$, then $|\xi'| \geq 1$. 

We are thus led to study the differential equation
\begin{equation}
e^{2k\xi}(1 - |\xi'|^2)^{k-1}[\xi'' + \frac{n-2k}{2k}(1-|\xi'|^2)] = \frac{(-1)^k n}{2k}.
	\label{Eq:xieq1}
\end{equation}
under the constraint that $|\xi'| > 1$.

It is well known (see \cite{C-H-Y, Viac00-Duke}) that \eqref{Eq:xieq1} has a first integral, namely
\[
H(\xi,\xi') := e^{(2k-n)\xi}(1 - |\xi'|^2)^k - (-1)^k e^{-n\xi} \text{ is (locally) constant along $C^2$ solutions.}
\]
A plot of the contours of $H$ for $k = 2, n = 7$ is provided in Figure \ref{Fig1}. See \cite{C-H-Y} for a more complete catalog. 

\begin{figure}[h]
\begin{center}
\includegraphics[scale=0.4]{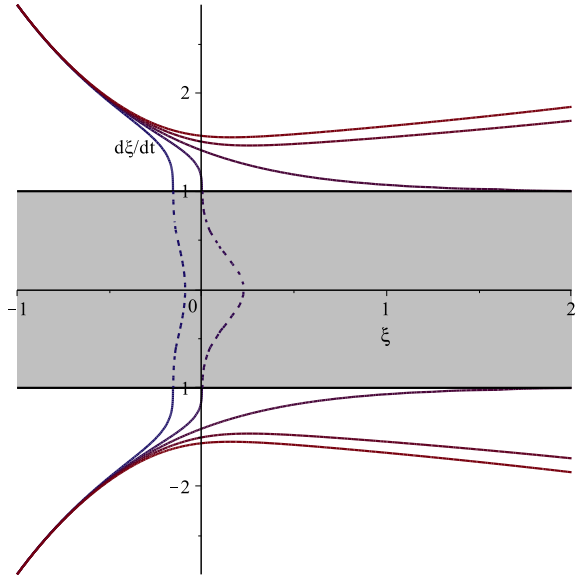}
\end{center}
\caption{The contours of $H$ for $k = 2$, $n = 7$. Each radially symmetric viscosity solution to \eqref{Eq:X1} lies on a single contour of $H$ but avoid the shaded region, i.e. the dotted parts of the contours of $H$ are excluded. Every smooth solution stays on one side of the shaded region. Every non-smooth solution jumps (on one contour) from the part below the shaded region to the part above the shaded region at a single non-differentiable point.}
\label{Fig1}
\end{figure}

Before moving on with the proofs of our results, we note the following statement.
\begin{rem}
As a consequence of Theorem \ref{thm:MainLipNegCab}, we have in fact that $H(\xi,\xi')$ is (locally) constant along viscosity solutions. 
\end{rem}

\begin{proof}
Fix $\tilde a < \tilde b$ in the domain of $u$ and apply Theorem \ref{thm:MainLipNegCab} relative to the interval $[\tilde a,\tilde b]$ with $c_1 = u(\tilde a)$ and $c_2 = u(\tilde b)$. If we are in cases 1--3, $u$ is $C^2(\tilde a,\tilde b)$ and so $H(\xi,\xi')$ is constant in $\{\tilde a < r < \tilde b\}$. Suppose we are in case 4. We have that $u$ is $C^2$ in $(\tilde a,m) \cup (m,\tilde b)$ for some $m$ and so $H(\xi,\xi')$ is constant in each of $\{\tilde a < r < m\}$ and $\{m < r < \tilde b\}$. Also, as $u$ is $C^1$ in each of $(\tilde a,m]$ and $[m,\tilde b)$, we have by assertion (iii) in case 4 that
\[
\lim_{r \rightarrow m^-} H(\xi(t),\xi'(t)) = H(\xi(t(m)),1) = H(\xi(t(m)),-1) = \lim_{r \rightarrow m^+} H(\xi(t),\xi'(t)).
\]
Hence $H(\xi,\xi')$ is also constant in $\{\tilde a < r < \tilde b\}$.
\end{proof}

\subsection{A lemma}
\label{ssec:lemma}

\begin{lem}\label{Lem:KeyLem}
For any $T > 0$, there exists a unique classical solution $\xi \in C^\infty(0,T) \cap C^{1,\frac{1}{k}}_{\rm loc}([0,T))$ to \eqref{Eq:xieq1} in $(0,T)$ such that
\begin{align}
&\lim_{t \rightarrow T^-} \xi(t) = -\infty,
	\label{Eq:xieq2}\\
&\xi'(0) = -1,\quad
	\xi'(t) < -1 \text{ in } (0,T).
	\label{Eq:xieq3}
\end{align}
Furthermore, for every $\gamma \in (\frac{1}{k}, 1]$, $\xi \notin C^{1,\gamma}_{\rm loc}([0,T))$.
\end{lem}

\begin{proof} We use ideas from \cite{C-H-Y}.

\medskip
\noindent
{\it Step 1:} We start by collecting relevant facts from \cite{C-H-Y} about the classical solution $\xi_{p,q}$ to \eqref{Eq:xieq1} satisfying the initial condition $\xi_{p,q}(0) = p$ and $\xi_{p,q}'(0) = q$ for $p \in \RR,q  \in (-\infty,-1)$  on its maximal interval of unique existence $I_{p,q} = (\underline{T}_{p,q}, \overline{T}_{p,q}) \subset \RR$.

Note that, since $\xi_{p,q}'(0) = q < -1$, it follows from \eqref{Eq:xieq1} that, for as long as $\xi_{p,q}$ remains $C^2$, $\xi_{p,q}' < -1$. Thus, as $H(\xi_{p,q}, \xi_{p,q}') = H(p,q)$, we have in $I_{p,q}$ that
\begin{equation}
 \xi_{p,q}'  = -\Big\{1 + e^{-2\xi_{p,q}}\big[1 +  (-1)^k H(p,q)e^{n\xi_{p,q}}\big]^{1/k}\Big\}^{1/2}.
 	\label{Eq:xipq'}
\end{equation}

By \cite{C-H-Y}(Theorem 1, Cases II.2 and II.3 for even $k$ and Theorem 2, Cases II.2 and II.3 for odd $k$), we have that $\overline{T}_{p,q}$ is finite (corresponding to $r_+$ being finite in the notation of \cite{C-H-Y}). Furthermore, 
\begin{equation}
\lim_{t \rightarrow \overline{T}_{p,q}^-} \xi_{p,q}(t) = -\infty.
	\label{Eq:oTend}
\end{equation}
By \eqref{Eq:xipq'} we thus have
\begin{equation}
\overline{T}_{p,q}
	= \int_{-\infty}^{p}
	\Big\{1 + e^{-2\xi}\big[1 - |H(p,q)| e^{n\xi}\big]^{1/k}\Big\}^{-1/2}\,d\xi.
		\label{Eq:oTpqExpr}
\end{equation}

In this proof, we will only need to consider the case that $(-1)^k H(p,q) < 0$. Then by \cite{C-H-Y}(Theorem 1, Case II.2 for even $k$ and Theorem 2, Case II.2 for odd $k$), we have that $\underline{T}_{p,q}$ is also finite (corresponding to $r_-$ being finite in the notation of \cite{C-H-Y}) and
\begin{equation}
\lim_{t \rightarrow \underline{T}_{p,q}^+} \xi_{p,q}(t) \text{ is finite}, \quad  \lim_{t \rightarrow \underline{T}_{p,q}^+} \xi_{p,q}'(t) = -1, \text{ and } \lim_{t \rightarrow \underline{T}_{p,q}^+} \xi_{p,q}''(t) = -\infty.
	\label{Eq:lTend}
\end{equation}

Using  \eqref{Eq:lTend} as well as the fact that $H(\xi_{p,q}, \xi_{p,q}') = H(p,q)$ and $\xi_{p,q}$ is decreasing, we have in $I_{p,q}$ that
\begin{equation}
\xi_{p,q} < \lim_{t \rightarrow \underline{T}_{p,q}^+} \xi_{p,q}(t) = -\frac{1}{n} \ln |H(p,q)| .
	\label{Eq:09I20-LVal}
\end{equation}
Differentiating \eqref{Eq:xipq'}, we see that, as $t \rightarrow \underline{T}_{p,q}^+$, 
\[
\lim_{t \rightarrow \underline{T}_{p,q}^+} (t - \underline{T}_{p,q})^{\frac{k-1}{k}}\xi_{p,q}''(t) \text{ exists and belongs to } (-\infty,0).
\]
Thus $\xi_{p,q}$ extends to a $C^{1,\frac{1}{k}}$ function in a neighborhood of $\underline{T}_{p,q}$ and $\xi_{p,q}$ does not extend to a $C^{1,\gamma}$ function in any neighborhood of $\underline{T}_{p,q}$.

Before moving on to the next stage, we note that, in view of \eqref{Eq:xipq'},
\begin{align}
\overline{T}_{p,q} - \underline{T}_{p,q} 
	&= \int_{-\infty}^{-\frac{1}{n} \ln |H(p,q)| }
	\Big\{1 + e^{-2\xi}\big[1 - |H(p,q)| e^{n\xi}\big]^{1/k}\Big\}^{-1/2}\,d\xi\nonumber\\
	&= \int_{-\infty}^{0}
	\Big\{1 +  |H(p,q)|^{\frac{2}{n}} e^{-2\eta }\big[1 -  e^{n\eta}\big]^{1/k}\Big\}^{-1/2}\,d\eta.
		\label{Eq:LengthIpq}
\end{align}
In particular, then length of $I_{p,q}$ depends only on $n$, $k$ and the value of $H(p,q)$, rather than $p$ and $q$ themselves.

\bigskip
\noindent{\it Step 2:} We now define for each given $p \in \RR$ a unique classical solution $\xi_p$ to \eqref{Eq:xieq1} in some maximal interval $(0,T_p)$ satisfying $\xi_p(0) = p, \xi_p'(0) = -1$ and $\xi_p' < -1$ in $(0,T_p)$.

It is clear that $(-1)^k H(p,-1) = -e^{-np} < 0$, and as $\partial_p H(p,-1) = (-1)^k n e^{-np} \neq 0$. By the implicit function theorem, there exist $\tilde p$ and $\tilde q < -1$ such that $H(\tilde p, \tilde q) = H(p,-1)$. Note that this implies
\[
-e^{-np} = (-1)^k H(\tilde p, \tilde q) > -e^{-n\tilde p} \text{ and so } \tilde p < p.
\]

Let
\[
\xi_p(t) = \xi_{\tilde p, \tilde q}(t + \underline{T}_{\tilde p, \tilde q}) \text{ and } T_p = \overline{T}_{\tilde p, \tilde q} - \underline{T}_{\tilde p, \tilde q}.
\]
By Step 1, it is readily seen that $\xi_p$ is smooth in $(0,T_p)$, belongs to $C^{1,\frac{1}{k}}_{\rm loc}([0,T_p))$ and no $C^{1,\gamma}_{\rm loc}([0,T_p))$ with $\gamma > \frac{1}{k}$, satisfies \eqref{Eq:xieq1} and $\xi_p' < -1$ in $(0,T_p)$,
\begin{align}
&\lim_{t \rightarrow T_p^-} \xi_{p}(t) = -\infty,
	\label{Eq:Tpend}\\
&\xi_p(0) =  -\frac{1}{n}\ln|H(\tilde p, \tilde q)| = p, \quad \xi_p'(0) = -1,
	\label{Eq:0enda}\\
&0 > \xi_{p,q}''(t) = O(t^{-\frac{k-1}{k}}) \text{ as } t \rightarrow 0^+,
	\label{Eq:0endb}\\
&\text{ and } T_p = \int_{-\infty}^{0}
	\Big\{1 + e^{-2\eta - 2p}\big[1 - e^{n\eta}\big]^{1/k}\Big\}^{-1/2}\,d\eta.
	\label{Eq:TpInt}
\end{align}

We claim that $\xi_p$ is unique in the sense that if $\hat \xi_p \in C^2(0,\hat T_p) \cap C^1([0,\hat T_p))$ is a solution to \eqref{Eq:xieq1} in some maximal interval $(0,\hat T_p)$ satisfying $\hat\xi_p(0) = p, \hat\xi_p'(0) = -1$ and $\hat\xi_p' < -1$ in $(0,\hat T_p)$, then $T_p = \hat T_p$ and $\xi_p \equiv \hat \xi_p$. To see this, note that, $\hat \xi_p(t) = \xi_{\hat \xi_p(s),\hat\xi_p'(s)}(t - s)$ for all $t, s \in (0,\hat T_p)$, since they both satisfy the same ODE in $t$ and agree up to first derivatives at $t = s$. By Step 1, $\hat\xi_p(t) \rightarrow -\infty$ as $t \rightarrow \hat T_p^{-}$, and so, as $\tilde p < p$ and $\hat \xi_p(0) = p$, there exists $t_0 \in (0,\hat T_p)$ such that $\hat \xi_p(t_0) = \tilde p$. This implies that $H(\tilde p, \hat\xi_p'(t_0)) = H(\hat\xi_p,\hat\xi_p') = H(p,-1) = H(\tilde p, \tilde q)$ and so $\hat \xi_p'(t_0) = \tilde q$. We deduce that $t_0 = - \underline{T}_{\tilde p, \tilde q}$, $\hat  T_p = T_p$ and $\hat \xi_p \equiv \xi_{\tilde p,\tilde q}(\cdot - t_0) \equiv \xi_p$, as claimed.

\bigskip
\noindent{\it Step 3:} From \eqref{Eq:TpInt}, we see that, as a function of $p$, $T_p$ is continuous and increasing and satisfies
\[
\lim_{p \rightarrow - \infty} T_p = 0 \text{ and } \lim_{p \rightarrow \infty} T_p = \infty.
\]
Thus, for any given $T > 0$, there is a unique $p(T)$ such that $T_{p(T)} = T$. The solution $\xi_{p(T)}$ to \eqref{Eq:xieq1} gives the desired solution.
\end{proof}

\subsection{Proof of Theorem \ref{thm:MainLipNeg}}
\label{ssec:Main}

Let $T = \frac{1}{2}\ln \frac{b}{a}$ and $t = \ln r - \frac{1}{2}\ln(ab)$. We need to exhibit a function $\xi: (-T, T) \rightarrow \RR$ such that $\xi$ is smooth in each of $(0,T)$ and $(-T, 0)$, is $C^{1,\frac{1}{k}}_{\rm loc}$ but not $C^{1,\gamma}_{\rm loc}$ for any $\gamma > \frac{1}{k}$ in each of $[0,T)$ and $(-T,0]$, the function $u$ defined by
\[
u(r) = \exp\Big[-\frac{n-2}{2}\Big(\xi(t) + \ln r\Big)\Big]
\]
solves \eqref{Eq:X1}-\eqref{Eq:X1BC} in $\{a < r = |x| < b\}$ in the viscosity sense, and
\begin{enumerate}[(i)]
\item $\lim_{t \rightarrow \pm T} \xi(t) = -\infty$,
\item $\xi'(0^-) = 1$, $\xi'(0^+) = -1$,
\item and $|\xi'| > 1$ in $(-T,0) \cup (0,T)$.
\end{enumerate}
Indeed, let $\xi^T:[0,T) \rightarrow \RR$ be the solution obtained in Lemma \ref{Lem:KeyLem}, and define
\[
\xi(t) = \left\{\begin{array}{ll}
	\xi^T(t) & \text{ if } 0 \leq t < T,\\
	\xi^T(-t) & \text{ if } -T < t < 0.
\end{array}\right.
\]
It is clear that $\xi$ satisfies all the listed requirements except for the statement that $u$ satisfies \eqref{Eq:X1} in the viscosity sense at $r =  \sqrt{ab}$. It remains to demonstrate, for any given $x_0$ with $|x_0| = \sqrt{ab}$, that
\begin{enumerate}[(a)]
\item if $\varphi$ is $C^2$ near $x_0$ and satisfies $\varphi \geq u$ near $x_0$ and $\varphi(x_0) = u(x_0)$, then $\lambda(-A^\varphi(x_0)) \in \Gamma_k$ and $\sigma_k(\lambda(-A^\varphi(x_0))) \geq 2^{-k} \big(\begin{array}{c}n\\k\end{array}\big)$,
\item and if $\varphi$ is $C^2$ near $x_0$ and satisfies $\varphi \leq u$ near $x_0$ and $\varphi(x_0) = u(x_0)$, then either $\lambda(-A^\varphi(x_0)) \notin \Gamma_k$ or $\lambda(-A^\varphi(x_0)) \in \Gamma_k$ but $\sigma_k(\lambda(-A^\varphi(x_0))) \leq 2^{-k} \big(\begin{array}{c}n\\k\end{array}\big)$.
\end{enumerate}

Without loss of generality, we may assume that $x_0 = (\sqrt{ab}, 0, \ldots, 0)$. 

Since $\partial_r \ln u|_{r = \sqrt{ab}^-} = -\frac{n-2}{\sqrt{ab}}<  0= \partial_r \ln u|_{r = \sqrt{ab}^+}$, there is no $C^2$ function $\varphi$ such that $\varphi \geq u$ near $x_0$ and $\varphi(x_0) = u(x_0)$. Therefore (a) holds.

Suppose now that $\varphi$ is a $C^2$ function such that $\varphi \leq u$ near $x_0$ and $\varphi(x_0) = u(x_0)$. As $u$ is radial, this implies that
\begin{align}
-\frac{n-2}{\sqrt{ab}} 
	&= \partial_{x_1} \ln u|_{r = \sqrt{ab}^-} \leq \partial_{x_1} \ln \varphi (x_0) \leq \partial_{x_1} \ln u|_{r = \sqrt{ab}^+} = 0,
		\label{Eq:Y1}\\
\partial_{x_2} \ln \varphi (x_0)
	&= \ldots = \partial_{x_n} \ln \varphi (x_0) = 0,
	\label{Eq:Y2}\\
\Big(\partial_{x_i}\partial_{x_j} \varphi(x_0) &- \frac{1}{\sqrt{ab}} \partial_{x_1}\varphi(x_0) \delta_{ij}\Big)_{2 \leq i,j \leq n} \leq 0.
	\label{Eq:Y3}
\end{align}
(For \eqref{Eq:Y3}, note that the matrix on the left hand side is the Hessian of $\varphi|_{\partial B_{\sqrt{ab}}}$ with respect to the metric induced on $\partial B_{\sqrt{ab}}$ by the Euclidean metric.) Now define $\bar \varphi(x) = \bar \varphi(|x|) = \varphi(|x|,0, \ldots, 0)$, $t = \ln r - \frac{1}{2} \ln(ab)$ and $\bar \xi(t) = -\frac{2}{n-2}\ln \bar\varphi(r) - \ln r$. By \eqref{Eq:Y1}, we have that $|\frac{d\bar\xi}{dt}(0)|\leq 1$ and so $\lambda(-A^{\bar\varphi}(x_0)) \notin \Gamma_k$. 

Let $O$ denote the diagonal matrix with diagonal entries $1, -1, \ldots, -1$. Note that, in block form,
\[
\nabla^2 \varphi(x_0) + O^t \nabla^2 \varphi(x_0)O = 2\left(\begin{array}{c|c}
\partial_{x_1}^2 \varphi(x_0) & 0\\
\hline
0& \big(\partial_{x_i}\partial_{x_j} \varphi(x_0) \big)_{2 \leq i,j \leq n}
\end{array}\right).
\]
Thus, by \eqref{Eq:Y3},
\[
\nabla^2 \varphi(x_0) + O^t \nabla^2 \varphi(x_0)O \leq 2\left(\begin{array}{c|c}
\partial_{x_1}^2 \varphi(x_0) & 0\\
\hline
0& \frac{1}{\sqrt{ab}} \partial_{x_1} \varphi(x_0) (\delta_{ij})_{2 \leq i,j \leq n}
\end{array}\right) = 2\nabla^2 \bar\varphi(x_0).
\]
Also, $\varphi(x_0) = \bar\varphi(x_0)$ and, in view of \eqref{Eq:Y2}, $\nabla\varphi(x_0) = \nabla\bar\varphi(x_0)$. Hence
\[
-A^{\varphi}(x_0) - O^tA^{\varphi}(x_0)O \leq - 2A^{\bar\varphi}(x_0).
\]
As $\lambda(-A^{\bar\varphi}(x_0)) \notin \Gamma_k$, it follows that $\lambda(-A^{\varphi}(x_0)- O^tA^{\varphi}(x_0)O) \notin\Gamma_k$. Since the set of matrices with eigenvalues belonging to $\Gamma_k$ is a convex cone (see e.g. \cite[Lemma B.1]{LiNgGreen}), we thus have that $\lambda(-A^{\varphi}(x_0)) \notin\Gamma_k$ or $\lambda(- O^tA^{\varphi}(x_0)O) \notin\Gamma_k$. Since $O$ is orthogonal, we deduce that $\lambda(-A^{\varphi}(x_0)) \notin\Gamma_k$. We have verified (b) and thus completed the proof.\hfill$\Box$

\subsection{Proof of Theorem \ref{thm:MainLipNegCab}}
\label{ssec:Cab}
As mentioned before, the uniqueness of solution follows from \cite{GonLiNg, LiNgWang}. We proceed to construct a radially symmetric solution with the indicated properties.

Let $T = \frac{1}{2}\ln \frac{b}{a}$, $p_a = -\frac{2}{n-2}\ln \consta - \ln a$ and $p_b = -\frac{2}{n-2}\ln \constb - \ln b$. We will only consider the case that $p_a \geq p_b$ (which is equivalent to $b^{\frac{n-2}{2}}\constb \geq a^{\frac{n-2}{2}}\consta$). (The case $p_a < p_b$ can be treated using an inversion about $|x| = \sqrt{ab}$.) We then have
\begin{align*}
T(a,b,\consta,\constb)	
	&=\frac{1}{2}\int_{p_b- p_a}^{0} \Big\{1 + e^{-2\eta-2p_a}\big[1 - e^{n\eta}\big]^{1/k}\Big\}^{-1/2}\,d\eta\\
	&=\frac{1}{2}\int_{p_b}^{p_a}
	\Big\{1 + e^{-2\xi}\big[1 - e^{n(\xi - p_a)}\big]^{1/k}\Big\}^{-1/2}\,d\xi
\end{align*}

\noindent(i) Suppose that $T < T(a,b,\consta,\constb)$. We show that Case 1 holds.
\medskip

Note that $H(p_a,-1) = -(-1)^k e^{-np_a}$. Thus as $T < T(a,b,\consta,\constb)$ and $(-1)^kH(p_a,\cdot)$ is decreasing in $(-\infty,-1)$, we can find $q_a < -1$ such that
\begin{equation}
T
	= \frac{1}{2}\int_{p_b}^{p_a}
		\Big\{1 + e^{-2\xi}\big[1 + (-1)^k H(p_a,q_a) e^{n\xi}\big]^{1/k}\Big\}^{-1/2}\,d\xi.
	\label{Eq:qaDef}
\end{equation}

Recall the solution $\xi_{p_a,q_a}$ to \eqref{Eq:xieq1} considered in the proof of Lemma \ref{Lem:KeyLem}. By \eqref{Eq:oTpqExpr}, we have that $2T < \overline{T}_{p_a,q_a}$. We then deduce from \eqref{Eq:xipq'} and \eqref{Eq:qaDef} that
\[
\xi_{p_a,q_a}(2T) = p_b.
\]
It thus follows that $\xi(t) = \xi_{p_a,q_a}(t + T)$ is smooth in $[-T,T]$, satisfies \eqref{Eq:xieq1} and $\xi' < -1$ in $(-T,T)$, as well as $\xi(-T) = p_a$ and $\xi(T) = p_b$. Returning to $u = \exp\big(-\frac{n-2}{2}\big(\xi(\ln r - \frac{1}{2}\ln(ab)) + \ln r\big)$ we obtain the conclusion.

\medskip 
\noindent(ii) Suppose that $T = T(a,b,\consta,\constb)$. We show that Case 3 holds.
\medskip

Recalling the definition of $T(a,b,\consta,\constb)$, we see that as $T > 0$, we have $p_a \neq p_b$. As $p_a \geq p_b$, we have $p_a > p_b$. We can now follow the argument in (i) with $\xi_{p_a, q_a}$ replaced by $\xi_{p_a}$ (defined in the proof of Lemma \ref{Lem:KeyLem}) to reach the conclusion. We omit the details.

\medskip
\noindent(iii) Suppose that $T > T(a,b,\consta,\constb)$. We show that Case 4 holds.
\medskip

In this case, we select $p \geq p_a (\geq p_b)$ such that
\begin{align*}
T 
	&= \frac{1}{2}\int_{p_b- p}^{0} \Big\{1 + e^{-2\eta-2p}\big[1 - e^{n\eta}\big]^{1/k}\Big\}^{-1/2}\,d\eta\\
		&\quad
	+  \frac{1}{2}\int_{p_a- p}^{0} \Big\{1 + e^{-2\eta-2p}\big[1 - e^{n\eta}\big]^{1/k}\Big\}^{-1/2}\,d\eta
\end{align*}
Such $p$ exists as the right hand side tends to $T(a,b,\consta,\constb)$ when $p \rightarrow p_a$ and diverges to $\infty$ as $p \rightarrow \infty$. Recall the solution $\xi_p$ defined in the proof of Lemma \ref{Lem:KeyLem}. Let 
\[
T_+ =  \frac{1}{2}\int_{p_b- p}^{0} \Big\{1 + e^{-2\eta-2p}\big[1 - e^{n\eta}\big]^{1/k}\Big\}^{-1/2}\,d\eta
\]
and
\[
T_- =  \frac{1}{2}\int_{p_a- p}^{0} \Big\{1 + e^{-2\eta-2p}\big[1 - e^{n\eta}\big]^{1/k}\Big\}^{-1/2}\,d\eta.
\]
Then $2T_\pm < T_p$ and the function $\xi_p$ satisfies $\xi_p(2T_+) = p_b$ and $\xi_p(2T_-) = p_a$. 

We then let
\[
\xi(t) = \left\{\begin{array}{ll}
	\xi_p(T_+ - T_- + t ) & \text{ if } - T_+ + T_-  \leq t < T,\\
	\xi_p(-T_+ + T_- - t) & \text{ if } -T < t < - T_+ + T_- .
\end{array}\right.
\]
We can then proceed as in the proof of Theorem \ref{thm:MainLipNeg} to show that $\xi$ is the desired solution.\hfill$\Box$


\appendix
\section{Appendix: Proof of Theorem \ref{Thm:GV}}

We abbreviate $u^{\frac{4}{n-2}}g$ as $g_u$. For small $\tau > 0$, let 
\begin{align*}
A_{g_u}^\tau 
	&= A_{g_u} + \tau \textrm{tr}_{g_u}(A_{g_u})g_u\\
	&= -\frac{2}{n-2} u^{-1}\,(\nabla_g^2 u + \tau \Delta_g u\,g) + \frac{2n}{(n-2)^2} u^{-2} d u \otimes d u - \frac{2}{(n-2)^2}\,u^{-\frac{2n}{n-2}}\,|\nabla_g u|_g^2\,g \\
		&\qquad 
	 + A_g + \frac{\tau}{2(n-1)} R_g\,g.
\end{align*}
By \cite[Theorem 1.4]{Gursky-Viaclovsky:negative-curvature}, we have for all sufficiently small $\tau > 0$ that the problem
\begin{equation}
\sigma_k\Big(\lambda\Big(-A_{g_{u_\tau}}^\tau\Big)\Big) 
	= 2^{-k} \Big(\begin{array}{c}n\\k\end{array}\Big), 
	\quad \lambda\Big(-A_{g_{u_\tau}}^\tau\Big) \in \Gamma_k, \quad u_\tau  > 0 \quad  \text{ in } M,
		\label{Eq:X1Mnfdtau}
\end{equation}
has a unique smooth solution $u_\tau$. Furthermore, by \cite[Propositions 3.2 and 4.1]{Gursky-Viaclovsky:negative-curvature}, the family $\{u_\tau\}$ is bounded in $C^1(M)$ as $\tau \rightarrow 0$. ($C^2$ bounds for $u_\tau$ were also proved in \cite{Gursky-Viaclovsky:negative-curvature}, but these bounds are unbounded as $\tau \rightarrow 0$.) Hence, along some sequence $\tau_i \rightarrow 0$, $u_{\tau_i}$ converges uniformly to some $u \in C^{0,1}(M)$. To conclude, we show that $u$ is a viscosity solution to \eqref{Eq:X1Mnfd}.

For notational convenience, we rename $u_{\tau_i}$ as $u_i$.

Fix some $\bar x \in M$.

\medskip
\noindent{\it Step 1:} We show that $u$ is a sub-solution to \eqref{Eq:X1Mnfd} at $\bar x$. More precisely, we show that for every $\varphi \in C^2(M)$ such that $\varphi \geq u$ on $M$ and $\varphi(\bar x) = u(\bar x)$ there holds that 
\begin{equation}
\lambda\Big(-A_{g_\varphi}(\bar x)\Big) \in \Big\{\lambda \in \Gamma_k\Big| \sigma_k(\lambda) \geq 2^{-k} \Big(\begin{array}{c}n\\k\end{array}\Big)\Big\} = \overline{S}_k =: \overline{S} .
	\label{Eq:02I20-E1}
\end{equation}

Here $d_g$ denotes the distance function of $g$ and $B_\delta(\bar x)$ denote the open geodesic ball of radius $\delta$ and centered at $\bar x$ with respect to $g$. Fix some arbitrary small $\delta > 0$ so that $\varphi_\delta := \varphi + \delta\,d_g(\cdot, \bar x)^2$ is $C^2$ in $\overline{B_\delta(\bar x)}$. 

Note that
\begin{equation}
\varphi_\delta = \varphi + \delta^3 \geq u + \delta^3 \text{ on } \partial B_\delta(\bar x) \text{ and } \varphi_\delta(\bar x) = u(\bar x).
	\label{Eq:02I20-E2}
\end{equation}
Select $x_{i,\delta} \in \overline{B_\delta(\bar x)}$ such that 
\[
(\varphi_\delta - u_i)(x_{i,\delta}) =  \inf_{B_\delta(\bar x)} (\varphi_\delta - u_i) =: m_{i,\delta}.
\]
By \eqref{Eq:02I20-E2} and the uniform convergence of $u_i$ to $u$, we have that $x_{i,\delta} \in B_\delta(\bar x)$. It follows that
\[
\nabla_g (\varphi_\delta - u_i)(x_{i,\delta}) = 0, \qquad \nabla_g^2 (\varphi_\delta - u_i)(x_{i,\delta}) \geq 0
\]
and so
\[
-A_{g_{\varphi_\delta - m_{i,\delta}}}^{\tau_i}(x_{i,\delta}) \geq -A_{g_{u_i}}^{\tau_i}(x_{i,\delta}).
\]
Recalling \eqref{Eq:X1Mnfdtau}, we hence have
\begin{equation}
\lambda\Big(-A_{g_{\varphi_\delta - m_{i,\delta}}}^{\tau_i}(x_{i,\delta})\Big) \in \overline{S}.
	\label{Eq:02I20-E3}
\end{equation}
On the other hand, as $\bar x$ is the unique minimum point of $\varphi_\delta - u$ in $\overline{B_\delta(\bar x)}$, we have $x_{i,\delta} \rightarrow \bar x$ and $m_{i,\delta} \rightarrow 0$ as $i \rightarrow \infty$. We can now pass $i \rightarrow \infty$ in \eqref{Eq:02I20-E3} to obtain
\[
\lambda\Big(-A_{g_{\varphi_\delta}}(\bar x)\Big) \in \overline{S}.
\]
Since $\delta$ is arbitrary, this proves \eqref{Eq:02I20-E1} after sending $\delta \rightarrow 0$.

\medskip
\noindent{\it Step 2:} We show that $u$ is a super-solution to \eqref{Eq:X1Mnfd} at $\bar x$, i.e. if $\varphi \in C^2(M)$ is such that $\varphi \leq u$ on $M$ and $\varphi(\bar x) = u(\bar x)$, then 
\begin{equation}
\lambda\Big(-A_{g_\varphi}(\bar x)\Big) \in \RR^n \setminus \Big\{\lambda \in \Gamma_k\Big| \sigma_k(\lambda) > 2^{-k} \Big(\begin{array}{c}n\\k\end{array}\Big)\Big\} = \underline{S}_k =: \underline{S}.
	\label{Eq:02I20-M1}
\end{equation}

The proof is analogous to that in Step 1. Fix some arbitrary small $\delta > 0$ so that $\hat \varphi_\delta := \varphi_{-\delta} = \varphi - \delta\,d_g(\cdot, \bar x)^2$ is $C^2$ in $\overline{B_\delta(\bar x)}$. Clearly
\[
\hat\varphi_{\delta}  \leq u - \delta^3 \text{ on } \partial B_\delta(\bar x) \text{ and } \hat\varphi_{\delta}(\bar x) = u(\bar x).
\]
We next select $\hat x_{i,\delta} \in \overline{B_\delta(\bar x)}$ such that 
\[
(\hat\varphi_{\delta} - u_i)(\hat x_{i,\delta}) =  \sup_{B_\delta(\bar x)} (\hat\varphi_{\delta} - u_i) =: \hat m_{i,\delta}.
\]
As before, we have $\hat x_{i,\delta} \in B_\delta(\bar x)$, $\nabla_g (\hat\varphi_{\delta} - u_i)(\hat x_{i,\delta}) = 0$, $\nabla_g^2 (\hat\varphi_{\delta} - u_i)(\hat x_{i,\delta}) \leq 0$ and 
\[
-A_{g_{\hat\varphi_{\delta} - \hat m_{i,\delta}}}^{\tau_i}(\hat x_{i,\delta}) \leq -A_{g_{u_i}}^{\tau_i}(\hat x_{i,\delta}).
\]
By \eqref{Eq:X1Mnfdtau}, we hence have
\begin{equation}
\lambda\Big(-A_{g_{\hat\varphi_{\delta} - \hat m_{i,\delta}}}^{\tau_i}(\hat x_{i,\delta})\Big) \in \underline{S}.
	\label{Eq:02I20-M3}
\end{equation}
Also, as $\hat x_{i,\delta} \rightarrow \bar x$ and $\hat m_{i,\delta} \rightarrow 0$ as $i \rightarrow \infty$, we can first pass $i \rightarrow \infty$ and then $\delta \rightarrow 0$ in \eqref{Eq:02I20-M3} to reach \eqref{Eq:02I20-M1}.
\hfill$\Box$


\newcommand{\noopsort}[1]{}

\end{document}